\documentclass{article}
\usepackage{latexsym}
\usepackage{verbatim}
\usepackage{amsmath}
\usepackage{amssymb}
\usepackage{amsthm}
\begin{document}

\def\Aut{\mathop{\rm Aut}\nolimits}
\def\Sym{\mathop{\rm Sym}\nolimits}
\def\Max{\mathop{\rm Max}\nolimits}
\def\an{\mathop{\rm an}\nolimits}
\def\SL{\mathop{\rm SL}\nolimits}

\def\Cl{\mathop{\rm Cl}\nolimits}
\def\val{\mathop{\rm val}\nolimits}
\def\Capa{\mathop{\rm Cap}\nolimits}
\def\Min{\mathop{\rm Min}\nolimits}
\def\prof{\mathop{\rm prof}\nolimits}
\def\Th{\mathop{\rm Th}\nolimits}
\def\lim{\mathop{\rm lim}\nolimits}
\def\Sup{\mathop{\rm Sup}\nolimits}
\def\rk{\mathop{\rm rk}\nolimits}

\newtheorem{defi}{Definition}[section]
\newtheorem{theorem}[defi]{Theorem}
\newtheorem{definition}[defi]{Definition}
\newtheorem{lemma}[defi]{Lemma}
\newtheorem{proposition}[defi]{Proposition}
\newtheorem{conjecture}[defi]{Conjecture}
\newtheorem{corollary}[defi]{Corollary}
\newtheorem{problem}[defi]{Problem}
\newtheorem{remark}[defi]{Remark}
\newtheorem{example}[defi]{Example}
\newtheorem{question}[defi]{Question}
\newtheorem{convention}[defi]{Conventions}
\newtheorem{fact}[defi]{Fact}

\title{Profinite groups with NIP theory and $p$-adic analytic groups}

\author{Dugald Macpherson\footnote{Research partially supported
by EPSRC grant EP/K020692/1},\\
School of Mathematics,\\ University of Leeds,\\Leeds LS2
9JT,\\England\\h.d.macpherson@leeds.ac.uk
 \and Katrin Tent\footnote{Research partially supported by SFB 878},\\Mathematisches Institut,\\ 
Universit\"at M\"unster,\\ 
Einsteinstrasse 62,\\
48149 M\"unster, 
Germany 
\\tent@math.uni-muenster.de}

\maketitle


\begin{abstract}
We consider profinite groups as 2-sorted first order structures, with a group sort, and a second sort which acts as an index set for a uniformly definable basis of neighbourhoods of the identity. It is shown that if the basis consists of {\em all} open subgroups, then the first order theory of such a structure is NIP (that is, does not have the independence property) precisely if the group has a normal subgroup of finite index which is a direct product of finitely many compact $p$-adic analytic groups, for distinct primes $p$. In fact, the condition NIP can here be weakened to NTP${}_2$. We also show that any NIP profinite group, presented as a 2-sorted structure, has an open prosoluble normal subgroup.

\end{abstract}

\section{Introduction}

We view a profinite group $G$ as an inverse limit of a given system $(H_i)_{i\in I}$ of finite groups, so equipped with a specified family $(K_i:i\in I)$ of open subgroups of finite index. We present $G$ as a structure $\mathcal{G}$ in a 2-sorted language $L_{\prof}$ with sorts $G,I$, with the group language on $G$, a partial order $\leq$ on $I$, and a relation $K\subset G\times I$ so that for each $i\in I$,
$K_i:=\{x\in G: \mathcal{G}\models K(x,i)\}$ is an open  subgroup of $G$, and $\{K_i:i\in I\}$ is a basis of neighbourhoods of 1 in $G$; for $i,j\in I$ we have $i\leq j$ if and only if $K_i\geq K_j$.  
We shall say that the 2-sorted profinite group
$\mathcal{G}=(G,I)$ is {\em full} if $\{K_i:i\in I\}$ consists of {\em all} the open subgroups of $G$. 

Given a complete first order theory $T$, a formula $\phi(\bar{x},\bar{y})$ has the {\em independence property} if
there is $M\models T$ and $\{\bar{b}_i:i\in {\mathbb N}\}\subset M^{|\bar{y}|}$ such that for any $S\subset {\mathbb N}$ there is $\bar{a}_S\in M^{|\bar{x}|}$ with $M\models \phi(\bar{a}_S,\bar{b}_i)$ if and only if $i\in S$. The theory $T$ is said to have the independence property if some formula has the independence property with respect to $T$, and to be {\em NIP}, or {\em dependent}, otherwise. This notion  is model-theoretically robust -- if a structure $N$ is interpretable in $M$ and $M$ has NIP theory, then so does $N$.

There is a notion of rank on partial types in finitely many variables, namely dp-rank, which takes ordinal values in NIP theories. For the definition, see \cite[Chapter 5]{simon}.   An NIP theory is {\em strongly NIP} if all types have finite dp-rank.

 The class of NIP theories properly contains that of stable theories, and has been the subject of considerable recent attention -- see e.g. \cite{simon}. In particular, a complete theory is NIP if and only if, in every model of $T$, any uniformly definable family of sets has finite {\em Vapnik-Chervonenkis dimension}, a notion important in statistical learning theory. There has been substantial work on groups definable in a NIP theory. Examples include stable groups, such as  abelian groups, algebraic groups over an algebraically closed field, and torsion-free hyperbolic groups;  and in the non-stable setting, the real and $p$-adic semi-algebraic Lie groups. We show here that also the compact $p$-adic analytic groups fit into the setting of NIP groups even when considered as 2-sorted full profinite groups.
 
 As a basic example, fix a prime $p$, let $G=({\mathbb Z}_p,+)$, let $I=\omega$, and define the relation $K$ so that for each $i\in \omega$, $K_i:=p^i{\mathbb Z}_p$. Then $G$ is $p$-adic analytic, $\mathcal{G}=(G,I)$ is a full pro-$p$ group, and $\mathcal{G}$ is interpretable in the $p$-adic field ${\mathbb Q}_p$ (with $I$ the non-negative part of the value group), so is NIP.

 Recall that, given a complete theory $T$, a formula $\phi(\bar{x},\bar{y})$ has TP${}_2$ (the tree property of the second kind) for $T$ if there is some $k\in {\mathbb N}$, some $M\models T$, and elements $\bar{a}_{ij}\in M^{|\bar{y}|}$ (for $i,j\in {\mathbb N}$) such that

(i) for each $i\in {\mathbb N}$, the formulas $\phi(\bar{x},\bar{a}_{ij})$ (for $j\in {\mathbb N}$) are $k$-inconsistent, that is, any conjunction of $k$ distinct such formulas is inconsistent with $T$, and

(ii) for any function $f:{\mathbb N}\to {\mathbb N}$, the set of formulas
$\{\phi(\bar{x},\bar{a}_{i,f(i)}):i\in {\mathbb N}\}$ is consistent with $T$.

\noindent
The theory $T$ is NTP${}_2$ if no formula has TP${}_2$ with respect to $T$. These theories encompass the NIP theories, see e.g. \cite{simon}, Prop. 5.3.1, and simple theories are also NTP${}_2$. Theories which are NTP${}_2$ have been extensively studied recently - see for example \cite{cks, hempel}.

Our main theorem is the following. 

\begin{theorem}\label{mainconj}
Let $\mathcal{G}=(G,I)$ be a  full profinite group. 
Then the following are equivalent:
\begin{enumerate}
\item $G$ has an open normal subgroup $N=P_1\times\ldots \times P_t$, where the $P_i$ are compact $p$-adic analytic groups with respect to distinct primes $p_1,\ldots,p_t$. 
\item  $\Th(\mathcal{G})$ is strongly NIP.
\item  $\Th(\mathcal{G})$ is NIP.
\item  $\Th(\mathcal{G})$ is NTP$_2$.
\end{enumerate}
\end{theorem}

The implication ($ 2.\Rightarrow 3.$) is trivial and ($ 3.\Rightarrow 4.$)  well-known, as pointed out above. It is thus left to show ($1.\Rightarrow 2.$) and ($4.\Rightarrow 1.$).
Part ($1. \Rightarrow  2.$) of Theorem~\ref{mainconj} consists essentially of the following proposition, which follows   from known results, especially those of du Sautoy in \cite{dus}. The structure ${\mathbb Z}_p^{{\rm an}}$ is an expansion of the $p$-adic field by restricted analytic functions, and is described at the start of Section 3. It was introduced by Denef and van den Dries in  \cite{dd}, who proved that it has quantifier elimination once an additional binary divisibility function is added to the language. For the definition of a {\em uniformly powerful} pro-$p$ group, see Section 2.

\begin{proposition} \label{dusautoy}
Let $\mathcal{G}$ be a full uniformly powerful pro-$p$  group. Then 
$\mathcal{G}$ is interpretable in the structure ${\mathbb Z}_p^{{\rm an}}$, 
whence $\Th(\mathcal{G})$ is strongly NIP.
\end{proposition}

We recall briefly some background. Following the usual convention, we say that a set $X\subset G$ (topologically) {\em generates} the profinite group $G$ if the abstract subgroup $\langle X\rangle$ (group-theoretically) generated by $X$ is dense in $G$, and say that $G$ is a {\em finitely generated} profinite group if $X$ can be chosen to be finite. If $G$ is a finitely generated profinite group then $d(G)$ denotes the cardinality of the smallest generating set for $G$, in the above topological sense. If $G$ is a profinite group and $X\subseteq G$, we write $\bar{X}$ for the topological closure of $X$ in $G$. The profinite group $G$ is said to have {\em rank $r$} if every closed subgroup of $G$ has a topological generating set of size $r$, and $r$ is minimal with this property. To avoid confusion with this rank, the model-theoretic rank mentioned earlier is always called {\em dp-rank}.

Note that any compact $p$-adic analytic group has the structure of a profinite group. In fact (see  \cite[Corollary 9.36]{ddms}) a topological group  is a compact $p$-adic analytic group if and only if it has an open subgroup of finite index which is a pro-$p$ group of finite rank.

It is immediate that every open subgroup of a profinite group has finite index. The following is a converse for finitely generated profinite groups, and was proved in the pro-$p$ case by Serre \cite[Section 4.2, Exercise 6]{serre} and in general by Nikolov and Segal \cite{NS1,NS2}. It is used occasionally in Section 4, but the uses are probably not essential. 

\begin{theorem} \label{ns}
Let $G$ be a finitely generated profinite group, and $H\leq G$. Then $H$ is open in $G$ if and only if $|G:H|$ is finite. 
\end{theorem}


\begin{example} \label{4ex} \rm We mention two similar-looking examples of profinite groups, which, viewed as 2-sorted full profinite groups, both fail to be NIP, but for slightly different reasons.  The first is the pro-$p$ group $G_1:=C_p {\rm wr} {\mathbb Z}_p$ which has a closed subgroup (the base group, a Cartesian power of $C_p$) which is not topologically finitely generated. Thus, as a full profinite group $G_1$ is not NIP, or even NTP${}_2$, by Lemma~\ref{l:NTP2}. The second is the following 2-generator group mentioned at the end of Chapter 5 of \cite{read}.  Let $p$ be a prime, and let $(q_i:i\in {\mathbb N})$ be a sequence of distinct primes such that $p^i|q_i-1$ for all $i$. Let $C$ be the Cartesian product of the cyclic groups $C_{q_i}$. Then ${\mathbb Z}_p$ has a natural action on $C$, and the profinite group $G_2=C\rtimes {\mathbb Z}_p$ has rank 2.
Now finite rank {\em pro-$p$}-groups are $p$-adic analytic (and so are NIP if presented as full profinite groups), but finite rank full profinite groups need not be NIP. The group $G_2$, presented as a {\em full} profinite group, is not NIP, by the Claim  in Section 4 below. Note that
$G_2$ is not virtually pronilpotent, but is prosoluble. It would be interesting to see whether either $G_1$ or $G_2$ could be presented as a 2-sorted (but {\em not} full) NIP profinite group.

On the other hand, compact $p$-adic analytic groups are all NIP when presented as full profinite groups. Examples include the  groups $({\mathbb Z}_p,+)$ and ${\rm SL}_2({\mathbb Z}_p)$, which are compact and {\em semi-algebraic}, that is, definable in the $p$-adic field ${\mathbb Q}_p$ in the language of rings.  For $p>2$, the graph of the $p$-adic 
exponential function $(p{\mathbb Z}_p,+)\to (1+p{\mathbb Z}_p,\cdot)$  is a compact $p$-adic analytic group. It is not semi-algebraic, but in the `analytic' expansion ${\mathbb Z}_p^{\an}$ of ${\mathbb Z}_p$ (see Section 3), it is  isomorphic to a semi-algebraic group, namely $({\mathbb Z}_p,+)$. 
\end{example}
 
 Section 2 of the paper contains some background on $p$-adic analytic groups, and we prove Proposition~\ref{dusautoy} in Section 3. Section 4 contains the remainder of the proof of Theorem~\ref{mainconj}. There is some further discussion in Section 5. We show there that {\em any} NIP profinite group (without a fullness assumption) has a prosoluble open subgroup of finite index (Proposition~\ref{prosol}).
 
\medskip

{\em Acknowledgement:} This research was partially supported by EPSRC grant EP/K020692/1
and SFB 878.

\section{Background}

We give a rapid introduction to $p$-adic analytic groups, taken from \cite[Chapter 9]{ddms}. First, if $V\subset {\mathbb Z}_p^r$ is non-empty and open, and $f=(f_1,\ldots,f_s):V\to {\mathbb Z}_p^s$ is a function, and $y\in V$, then $f$ is {\em analytic at $y$} if there are
$F_1,\ldots,F_s\in {\mathbb Q}_p[[X_1,\ldots,X_r]]$ for $i=1,\ldots,s$ and $h\in {\mathbb N}$ such that $f_i(y+p^hx)=F_i(x)$ for each $i=1,\ldots,s$ and $x\in {\mathbb Z}_p^r$. We say $f$ is {\em analytic on $V$} if it is analytic at every point of $V$. 

If $X$ is a topological space, then a {\em $p$-adic chart of dimension $n$} of $X$ is a triple $(U,\phi,n)$ where $U$ is an open subset of $X$, and $\phi$ is a homeomorphism from $U$ onto an open subset
 of ${\mathbb Z}_p^n$. The charts $c=(U,\phi,n)$ and  $d=(V,\psi,m)$ of $X$ are {\em compatible} if, putting $W:=U\cap V$, the maps $\psi \circ \phi^{-1}|_{\phi(W)}$ and 
$\phi\circ\psi^{-1}|_{\psi(W)}$
 are analytic on $\phi(W)$ and $\phi(V)$ respectively. There is a natural notion of ($p$-adic) {\em atlas} on the topological space $X$, consisting of a covering by compatible charts, and an 
equivalence relation of {\em compatibility} of atlases: two atlases $A$ and $B$ of $X$ are compatible if every chart of $A$ is compatible with every chart of $B$. Finally, a {\em ${\mathbb Q}_p$-analytic
 manifold structure} on $X$ is an equivalence class of compatible atlases on $X$, and such a structure is called a {\em $p$-adic analytic manifold}. 

If $G$ is a topological group, then $G$ is a {\em $p$-adic analytic group}, or {\em $p$-adic Lie group}, if $G$ has the structure of a $p$-adic analytic manifold  such that the maps $G\times G\to G$ and $G\to G$, given by group multiplication and inversion respectively, are analytic.

If $H$ is  a subgroup of the group $G$ and $n>0$, put $H^n:=\langle \{h^n:h\in H\}\rangle$. Recall that if $G$ is a pro-$p$-group, then $G$ has a series of closed normal subgroups (the `lower $p$-series')
$G=P_1(G) \geq P_2(G)\geq\ldots$, where $P_{n+1}(G):=\overline{P_{n}(G)^p[P_n(G),G]}$. If $G$ is a finitely generated pro-$p$-group theb the $P_i(G)$ form a base of open neighbourhoods of the identity -- see e.g. \cite[Proposition 1.16(iii)]{ddms}. The pro-$p$ group $G$ is said to be {\em powerful}  if $p$ is odd  and $G/\overline{G^p}$ is abelian, or if $p=2$ and $G/\overline{G^4}$ is abelian. It is {\em uniformly powerful} if (i) $G$ is finitely generated, (ii) $G$ is powerful, and (iii) $|P_i(G):P_{i+1}(G)|=|G:P_2(G)|$ for each $i$. A major  theorem of Lazard \cite{lazard} states that a compact topological group is a $p$-adic analytic group if and only if it has an open subgroup which is a uniformly powerful pro-$p$ group. Building on this, we have the following characterisations, combining  the work of Lazard, Lubotzky, Mann, Segal, and Shalev.

\begin{theorem} \label{lazard}
Let $G$ be a  pro-$p$ group. Then the following are equivalent. 

(i) $G$ is a compact $p$-adic analytic group.

(ii) $G$ is finitely generated and has a uniformly powerful subgroup of finite index.

(iii) $G$ has finite rank.

(iv) $G$ has polynomial subgroup growth, that is, there is $\alpha>0$ such that for each $n>0$, $G$ has at most $n^\alpha$ open subgroups of index $n$.

(v) $G$ is finitely generated and does not involve arbitrarily large wreath products of the form $C_p {\rm wr} C_{p^n}$.

(vi) $G$ is isomorphic to a closed subgroup of ${\rm GL}_d({\mathbb Z}_p)$ for some suitable $d$.

\end{theorem}

\begin{proof} See for example Theorem 5.11 of \cite{klopsch}, together with the main theorem of \cite{shalev1} for (v). 
\end{proof}

We will only be using the characterizations given in parts (i)--(iii).
Parts (iv)--(vi) of the last theorem are not used in this paper. Note that by Theorem~\ref{mainconj}, we may now add to Theorem~\ref{lazard} the model-theoretic characterisation

(vii) the full profinite group $\mathcal{G}=(G,I)$ is NIP.

\section{Proof of Proposition~\ref{dusautoy}.}

We first recall the main results of Denef and van den Dries  \cite{dd} on the structure ${\mathbb Z}_p^{\an}$. Let $X=(X_1,\ldots, X_m)$ consist of $m$ commuting indeterminates. If $\nu=(\nu_1,\ldots,\nu_m)\in {\mathbb N}^m$ is a multi-index, then we put $|\nu|:=\nu_1+\ldots+\nu_m$. Let ${\mathbb Z}_p\{X\}$ denote the ring of all formal power series
$\Sigma_{\nu\in {\mathbb N}^m} a_\nu X^\nu$ where $a_\nu\in {\mathbb Z}_p$ for each $\nu$ and $|a_\nu|_p\to 0$ as $|\nu|\to \infty$ (here $|a|_p$ denotes the $p$-adic norm of $a$). 

The language $L^{\an}_D$ contains for all $m\in\mathbb{N}$ an $m$-ary function symbol $F$ for each $F(X)\in {\mathbb Z}_p\{X\}$, a binary function symbol $D$, and a unary relation symbol $P_n$ 
for 
each $n>0$. We interpret ${\mathbb Z}_p$ as an $L^{\an}_D$-structure as follows: we interpret $P_n$ by the set of non-zero $n^{{\rm th}}$ powers for each $n$, the $m$-ary function symbol $F$ by the function 
induced by $F(X)$ (which converges on ${\mathbb Z}_p^m$), and we interpret $D:{\mathbb Z}_p^2\to {\mathbb Z}_p$ by putting
$D(x,y)=x/y$ if $|x|_p\leq |y|_p$ and $y\neq 0$, and $D(x,y)=0$ otherwise. Let $T^{\an}_D$ be the theory of ${\mathbb Z}_p$ in the language $L^{\an}_D$.
We remark that  by \cite[Lemma 1.9]{dus}, if $f:{\mathbb Z}_p^m\to {\mathbb Z}_p$ is an analytic function, then $f$ is definable in the language $L^{\an}_D$.
This is an elementary application of the topological compactness of ${\mathbb Z}_p^m$. 

By Theorem 1.1  of \cite{dd} the theory $T^{\an}_D$ has quantifier elimination.
For us the following is crucial:
\begin{theorem} \label{dd}
The theory $T^{\an}_D$ is strongly NIP.
\end{theorem}

\begin{proof} Let ${\mathbb Q}_p^{\an}$ denote the $p$-adic {\em field} equipped with subanalytic structure; the function symbols interpreting elements of ${\mathbb  Z}_p\{X\}$ take value $0$ on ${\mathbb Q}_p^m\setminus {\mathbb Z}_p^m$, so 
${\mathbb Q}_p^{\an}$ is interpretable in ${\mathbb Z}_p^{\an}$. By \cite[Theorem A]{dhm}, the theory of ${\mathbb Q}_p^{\an}$  is $P$-minimal, in the sense of \cite{hm}. By \cite[Proposition 7.1]{hm}, any $P$-minimal theory is NIP. 
By 7.3, 7.9 and 7.10 of \cite{vc}, the structure ${\mathbb Q}_p^{{\rm an}}$ is  dp-minimal, that is, of dp-rank 1, and so is strongly NIP.
\end{proof}


In order to prove Proposition~\ref{dusautoy} it suffices by Theorem~\ref{dd} to prove that if ${\cal G}$ is a uniformly powerful pro-$p$ group, then $G$ together with its family of open subgroups is  definable in $\mathbb{Z}_p^{\an}$. This is essentially contained in Section 2 of \cite{dus}, though it is not stated in this form, so we sketch the arguments from there. (Du Sautoy has a different goal, namely to show that certain Poincar\'e series enumerating subgroups of given finite index are rational.) 

Observe (\cite[Definition 1.11]{dus}) that if $G$ is a pro-$p$ group, then $G$ admits a natural action by $({\mathbb Z}_p,+)$: for $\lambda\in {\mathbb Z}_p$ and $g\in G$, put $g^\lambda:=\lim_{n\to \infty} g^{a_n}$, where $(a_n)$ is any sequence from ${\mathbb Z}$ with limit $\lambda$ (this is well-defined).
Also, we define $\omega:G\to {\mathbb N}\cup\{\infty\}$ by putting $\omega(g)=n$ if $g\in P_n(G)\setminus P_{n+1}(G)$, and $\omega(1)=\infty$. The function $\omega$ is analogous to a valuation on $G$.


We next state a key theorem from \cite{dus}.
\begin{theorem}[Theorem 1.18 of \cite{dus}]\label{dusrep}
Let $G$ be a uniformly powerful pro-$p$-group with $d(G)=d$, and let $\{x_1,\ldots,x_d\}$ be a topological generating set for $G$. Then:

(i) for each $x\in G$ there are unique $\lambda_1,\ldots,\lambda_d\in {\mathbb Z}_p$ such that if $\lambda=(\lambda_1,\ldots,\lambda_d)$, then 
$x=x_1^{\lambda_1}\cdots x_d^{\lambda_d}$ (and is denoted $x(\lambda)$);

(ii) the function $f: {\mathbb Z}_p^d \times {\mathbb Z}_p^d \to {\mathbb Z}_p^d$ determined by the rule  $x(\lambda).(x(\mu))^{-1}=x(f(\lambda,\mu))$ is an analytic function; 

(iii) if $x=x(\lambda)\in G$, then $\omega(x)= \Min\{v(\lambda_i)+1: 1\leq i\leq d\}$.
\end{theorem}

We summarize this as:
\begin{corollary}\label{c:defG}
Any uniformly powerful pro-$p$-group $G$ of rank $d$ is (isomorphic to) a group  definable in the $L^{\an}_D$-structure ${\mathbb Z}^{\an}_p$ with domain ${\mathbb Z}_p^d$.
Furthermore, the function  $\omega$ is definable in this structure.
\end{corollary}
\begin{proof}
Fix an ordered  topological generating set $(x_1,\ldots,x_d)$. By (i) of Theorem~\ref{dusrep},  we may identify $G$ with an isomorphic copy of $G$ with domain  ${\mathbb Z}_p^d$, with the group structure definable in ${\mathbb Z}_p^{{\rm an}}$ by (ii). By (iii), $\omega$ is also definable.
\end{proof}

We shall refer to an ordered  topological generating set for $G$ of size $d(G)$ as a {\em basis} of $G$. It acts as a system of coordinates for $G$.
It remains to show that the family of open subgroups is also uniformly definable. For this we first note

\begin{lemma} \label{dusexp}
Let $G$ be a uniformly powerful pro-$p$ group with a basis $\{x_1,\ldots,x_d\}$. Then
 the action (described above) of ${\mathbb Z}_p$ on $G$ by exponentiation is definable in ${\mathbb Z}_p^{{\rm an}}$.
\end{lemma}
\begin{proof}
The function $g:{\mathbb Z}_p^d\times {\mathbb Z}_p \to {\mathbb Z}_p^d$ defined by
$$x(\lambda_1,\ldots,\lambda_d)^\mu =x(g(\lambda,\mu))$$
is an analytic function on ${\mathbb Z}_p^d \times {\mathbb Z}_p$ (see \cite[Lemma 1.19]{dus}), so interprets a function symbol of $L^{\an}_D$ and so is definable.
\end{proof}

We shall put $G_n:=P_n(G)$ for each $n\geq 1$. For the following, see the discussion after \cite[Theorem 2.1]{dus}. The map $x\mapsto x^p$ yields an isomorphism  $f_n:G_n/G_{n+1} \to G_{n+1}/G_{n+2}$: if $x=x(\lambda) \in G_n$
(in the notation of Theorem~\ref{dusrep}) then $f_n(x(\lambda)G_{n+1})=x(p\lambda)G_{n+2}$. Let $\pi:{\mathbb Z}_p\to {\mathbb F}_p$ denote the residue map, and for each $n\geq 1$ define $\pi_n:G_n\to {\mathbb F}_p^d$ by putting
$$\pi_n(x(\lambda))=(\pi(p^{-(n-1)}\lambda_1), \ldots,\pi(p^{-(n-1)}\lambda_d)).$$ Then $\pi_n$ is a group homomorphism with kernel $G_{n+1}$, so induces a group isomorphism 
$G_n/G_{n+1}\to (({\mathbb F}_p )^d,+)$. 

\begin{definition}[Definition 2.2 of \cite{dus}]
Let $G$ be a uniformly powerful pro-$p$ group with $d(G)=d$. If $H$ is an open subgroup of $G$, then  the tuple $(h_1,\ldots,h_d)\in H^d$ is a {\em good basis} of $H$ if

(i) $\omega(h_i)\leq \omega(h_j)$ whenever $i\leq j$, and

(ii) for each $n\in {\mathbb N}$,  if $I_n:=\{j: \omega(h_j)=n\}$ then  $\{\pi_n(h_j):j\in I_n\}$  extends the linearly independent set $\{\pi_n(h_j^{p^{n-\omega(h_j)}}): j\in I_1 \cup \ldots \cup I_{n-1}\}$ to a basis of $\pi_n(H\cap P_n(G))$.
\end{definition}

With this definition we have the following:
\begin{lemma}[Lemma 2.5 of \cite{dus}]\label{l:goodbasis}
Let $G$ be a uniformly powerful pro-$p$ group with $d(G)=d$.
The tuple $(h_1,\ldots,h_d)\in G^d$ is a good basis for the  open subgroup $H$ of $G$ if and only if

(i) $\omega(h_i)\leq \omega(h_j)$ whenever $i\leq j$,

(ii) $h_i\neq 1$ for each $i=1,\ldots,d$,

(iii) $\{h_1^{\lambda_1}\cdots h_d^{\lambda_d}:\lambda_i\in {\mathbb Z}_p\}$ is a subgroup of $G$ (and equals $H$), 

(iv) for all $\lambda_1,\ldots,\lambda_d\in {\mathbb Z}_p$ we have $\omega(h_1^{\lambda_1}\cdot\ldots\cdot h_d^{\lambda_d})=\Min\{\omega(h_i)+\nu(\lambda_i):1\leq i\leq d\}$.\\
Furthermore, by \cite[Lemma 2.7]{dus}, every open subgroup of $G$ has a good basis.
\end{lemma}

We complete the proof of Proposition~\ref{dusautoy}. So let $G$ be a uniformly powerful pro-$p$ group with $d(G)=d$, and let $\mathcal{G}=(G,I)$ be the corresponding 2-sorted full profinite group. 
Since the map $\omega$ and the $\mathbb{Z}_p$-action on $G$ are   definable in ${\mathbb Z}_p^{\an}$, it follows from Lemma~\ref{l:goodbasis} that the set of all good bases of open subgroups of $G$ is definable in ${\mathbb Z}_p^{\an}$. 

Furthermore, again using Lemma~\ref{l:goodbasis}, the set of elements of an open subgroup $H$ is  definable in a uniform way from the good basis as a parameter. Thus, there is a definable equivalence relation on the collection of all good bases, with two good bases equivalent if they are good bases for the same open subgroup, and we may identify the index set $I$ with the set of equivalence classes of good bases.
Since the relation $K\subset G\times I$ is definable, the structure $\mathcal{G}=(G,I)$ is interpretable in ${\mathbb Z}^{\an}_p$ (over $L^{\an}_D$). Hence, by Theorem~\ref{dd}, the theory of $\mathcal{G}$ is strongly NIP.

\section{Proof of Theorem~\ref{mainconj}.}

As noted in the Introduction, the only non-trivial parts of the proof of Theorem~\ref{mainconj} are $1.\Rightarrow 2.$ and $4.\Rightarrow 1.$.

{\em Proof of Theorem~\ref{mainconj}\/($1.\Rightarrow 2.$)}
 Let $\mathcal{G}$ satisfy the condition (1.) in Theorem~\ref{mainconj}. By Lazard's Theorem (contained in Theorem~\ref{lazard}),
 any $p$-adic analytic group has a uniformly powerful pro-$p$ normal subgroup of finite index. Thus, we may suppose that $p_1,\ldots,p_t$ are distinct  primes, and that for each $i=1,\ldots,t$, there is a uniformly powerful pro-$p_i$-group $P_i$, and that $N:=P_1\times \ldots \times P_t$ is a normal subgroup of $G$ of finite index.  Let $M$ be the disjoint union of the rings ${\mathbb Z}_{p_i}^{\an}$ (viewed in a language in which the rings have formally disjoint languages).  Then $M$  has finite dp-rank by \cite[Theorem 4.8]{kou} and hence is strongly NIP.
By Proposition~\ref{dusautoy}, each $P_i$ is interpretable in ${\mathbb Z}_{p_i}^{\an}$ when viewed as a full 2-sorted profinite group. Furthermore, any open subgroup of  $N$ is a direct product of open subgroups of the $P_i$, essentially because the same statement holds in finite groups -- any finite nilpotent group is a direct product of its Sylow subgroups. Hence, easily, the full profinite group  $N$
 is interpretable in $M$, and thus strongly NIP.

Let $F:=G/N$. As a group, $G$ is determined by the pair $(N,F)$ together with a pair $(\mu,f)$, where 
$\mu:F\to \Aut(N)$  and $f:F\times F\to N$  are  functions -- see e.g. Section IV.6 of \cite{brown}. Since $F$ is finite, the map $f$ is definable in $(N,F)$ by naming finitely many constants. In addition, by \cite[Theorem 1.18 (iii)]{dus}, if $H$ is any uniformly powerful  pro-$p$ group  living on ${\mathbb Z}_p^d$ as in the last section, then any automorphism $\phi$ of $H$ corresponds to an analytic (so ${\mathbb Z}^{\an}_p$-definable) map $\Phi:{\mathbb Z}_p^d \to {\mathbb Z}_p^d$ given by
$$x(\lambda)^\phi=x(\Phi(\lambda)).$$
Thus, for each $g\in F$, the automorphism $\mu(g)$, being a tuple of such maps $\phi$,  is definable in the  strongly NIP structure $M$. It follows that the group $G$ is definable in a strongly NIP structure, namely the disjoint union of $M$ and the finite structure $F$.

It remains to check that the open subgroups of $G$ are uniformly definable. However, if $H$ is an open subgroup of $G$, then $H\cap N$ is an open subgroup of $N$ so is (uniformly) definable in the full profinite group $N$, and $H/H\cap N \cong HN/N\leq G/N$. 
Thus, there are $h_1,\ldots,h_e\in H$ (where $e\leq |G/N|$, so is bounded) such that
$$H=(H\cap N)h_1 \cup \ldots\cup (H\cap N)h_e.$$
Since the set of all tuples $(h_1,\ldots,h_e)$ which yield a subgroup of $G$ in this way is definable, and we can define when two such tuples yield the same group, the groups $H$ are uniformly definable in the structure $M$, as required. $\Box$

\medskip

{\em Proof of Theorem~\ref{mainconj} } ($4.\Rightarrow 1$.) 
First, recall the following well-known facts. If $H$ is a subgroup of the profinite group $G$ we write $H\leq_O G$ (respectively $H\leq_C G$) if $H$ is open (respectively closed) in $G$; we adopt a corresponding notation for normal subgroups. If $F$ is a finite group, then $\Phi(F)$ denotes the {\em Frattini subgroup} of $F$, that is, the intersection of the maximal subgroups of $F$. Extending earlier notation, if $G$ is a finite group then $d(G)$ denotes the smallest size of a generating set for $G$.

\begin{proposition} \label{genfacts}
(i) Let $G$ be a profinite group. Then 
$$\rk(G):=\Sup\{d(H): H\leq_C G\}=\Sup\{\rk(G/N): N\triangleleft_O G\}$$
(possibly infinite).

(ii) Let $G$ be a finite  group such that every Sylow subgroup can be generated by $d$ elements. Then $d(G)\leq d+1$.

(iii) Let $P$ be a finite $p$-group, with Frattini subgroup $\Phi(P)$. The $P/\Phi(P)$ is an elementary abelian $p$-group of rank $d(P)$.
\end{proposition}

\begin{proof}(i) See Proposition 3.11 of \cite{ddms}.

(ii) See the discussion after Proposition 8.2.4 of \cite{wilson}, and the references there to \cite{luc} and \cite{gur}.

(iii) This is the Burnside Basis Theorem, and is standard. 
\end{proof}

If $G$ is a profinite group, then $\Aut(G)$ denotes the group of all topological automorphisms of $G$. The profinite group $G$ is said to {\em virtually} have property $\mathcal{P}$ if some open normal subgroup of $G$ has property $\mathcal{P}$. By \cite[Theorem 5.3]{ddms}, if $G$ is a finitely generated profinite group then $\Aut(G)$ has the structure of a  profinite group. 

\begin{theorem} [Theorem 5.7 of \cite{ddms}] \label{aut} Let $G$ be a finitely generated profinite group. If $G$ is virtually a pro-$p$-group, then $\Aut(G)$ is also virtually a pro-$p$ group. 
\end{theorem}

We will need the following general lemma.

\begin{lemma}\label{l:NTP2}
Let $G$ be an $\emptyset$-definable  group in a structure with  NTP${}_2$ theory, and $\psi(x,\bar{y})$ a formula implying $x\in G$. Then there is  $k=k_\psi\in {\mathbb N}$ such that the following holds. 
Suppose that $H$ is a subgroup of $G$, $\pi: H\longrightarrow \Pi_{i\in J} T_i$ is an epimorphism to the Cartesian product of the groups $T_i$, and
$\pi_j: H\longrightarrow T_j$ is for each $j\in J$ the composition of $\pi$ with the canonical projection $\Pi_{i\in J} T_i\to T_j$.
Suppose also that for each $j\in J$,  there is a subgroup $\bar{R}_j\leq G$  and group  $R_j <T_j$  with
$\bar{R}_j\cap H=\pi_j^{-1}(R_j)$, such that finite intersections of the groups $\bar{R}_j$ are uniformly definable by instances of $\psi(x,\bar{y})$.
Then $|J|\leq k$.

If the underlying theory is NIP, it suffices (for the finite bound on $J$) that the $\bar{R}_j$ for $ j\in J$ are uniformly
definable.
\end{lemma}
\begin{proof}
 Suppose for a contradiction that this is false. Then for any $l\in {\mathbb N}$ we may find data $J, H, \bar{R}_j$ etc. as in the statement, such that $|J|\geq l^2$. Choose a partition $J=\bigcup_{i\in I} J_i$ of $J$ into finite sets $J_i$ each of size at least $l$, with $|I|\geq l$.
 
For each $i\in I$, let $\bar{S}_i:= \bigcap_{j\in J_i}\bar{R}_j$, and put $S_i:=\bar{S}_i\cap H=\bigcap_{j\in J_i}\pi_j^{-1}(R_j)$. By our assumption, for each $i\in I$ there is $\bar{a}_i$ such that $\bar{S}_i=\psi(G,\bar{a}_i)$. For each $j\in J$ choose $g_j\in H$ such that $\pi_j(g_j)\in T_j\setminus R_j$ and $\pi_i(g_j)=1$ for $i\neq j$. Note that for any $i\in I$, the elements $g_j$ with $j\in J_i$ all lie in distinct cosets of $S_i$ in $H$. Hence the cosets $\bar{S}_ig_j$ of $\bar{S}_i$ in $G$ are distinct (for distinct  $j\in J_i$) and are uniformly definable by some formula $\phi(x,\bar{a}_ig_j)$ with $\phi$ dependent only on $\psi$.

We claim that the formula $\phi(x,\bar{y}z)$ has TP${}_2$. Clearly, for any $i$, the formulas
$\phi(x,\bar{a}_ig_j)$ (for $j\in J_i$) are 2-inconsistent, since they define different cosets of the same group. 
We shall show that if $f(i)\in J_i$ for each $i\in I$, then the set $\{\phi(x, \bar{a}_ig_{f(i)}): i\in I\}$ is consistent. This will show that the formula $\phi(x,\bar{y}z)$ has TP${}_2$, for to show there 
are $\bar{a}_{ij}$ (for $i,j\in {\mathbb N}$) satisfying conditions (i) and (ii) of the definition of TP${}_2$ in Section 1, it suffices by compactness to find arbitrarily large finite such arrays ; we get these by putting  $\bar{a}_{ij}=(\bar{a}_i,g_j)$.

So let $i_1,\ldots,i_t\in I$. Observe that if $j\not\in J_i$ then $g_j\in S_i$. 
Let $h\in H$ be the unique element which projects to $g_{i,f(i)}$ in the $f(i)$-coordinate for each $i=1,\ldots, t$, and 
projects to the identity in other coordinates. Then $h\in S_ig_{f(i)}$ for each $i=1,\ldots,t$, so $\bigwedge_{i=1}^t \phi(h,\bar{a}_ig_{f(i)})$ holds, as required.

If the theory   is NIP and the $\bar{R}_j$ are uniformly definable by $\chi(x,\bar a_j)$, just pick $g_j\in \pi_j^{-1}(T_j\setminus R_j)$ for each  $j\in J$. For a finite subset $F\subseteq J$ let $h_F$ be the unique element of $H$ such that $\pi_j(h_F)=g_j$ if $j\in F$, and $\pi_j(h_F)=1$ otherwise. Then $h_F\in \bar{R_j}$ if and only if $j\not\in F$, contradicting the NIP assumption.

\end{proof}

\begin{lemma} \label{finrank} Let $\mathcal{G}=(G,I)$ be a 2-sorted full profinite group with NTP${}_2$ theory. Then $G$ has finite rank (in the sense of profinite groups).
\end{lemma}

\begin{proof}
Suppose for a contradiction that $G$ has infinite rank. Then by Proposition~\ref{genfacts}(i), for every $k$ there is $N_k\triangleleft_O G$ such that $\rk(G/N_k)\geq k+1$. Hence there is $H_k$ with $N_k\leq H_k\leq G$ such that $d(H_k/N_k)\geq k+1$. 
Thus, by Proposition~\ref{genfacts}(ii), there is a prime $p$ such that some Sylow $p$-subgroup $Q_k$ of $H_k/N_k$ satisfies $d(Q_k)\geq k$. 

Let $\pi_k:H_k\to H_k/N_k$ be the natural map, and put $\Phi_k:=\pi_k^{-1}(\Phi(Q_k/N_k))$. Then $\pi_k^{-1}(Q_k)/\Phi_k$ is an elementary abelian $p$-group of rank $l\geq k$. 

We now apply Lemma~\ref{l:NTP2} to $\pi_k^{-1}(Q_k)\longrightarrow \Pi_{i<l}(\mathbb{Z}/p\mathbb{Z})$  to obtain a contradiction, using  the fact that subgroups $P$ (where $N_k<P<G$ for some $k\in \mathbb{N}$) and their finite intersections are uniformly definable.


\end{proof}

{\em Proof of Theorem~\ref{mainconj}}( $4.\Rightarrow 1.$)

Now suppose that $\mathcal{G}=(G,I)$ is  a full profinite group with NTP${}_2$ theory. By Lemma~\ref{finrank}, $G$ has finite rank. By Theorem~\ref{ns}, we may now assume that the finite index subgroups of $G$ are open, and so are uniformly definable in $\mathcal{G}$. By Corollary 5.4.5 of \cite{read} (extending \cite[Theorem 8.4.1]{wilson}), $G$ has  normal subgroups
 $N\leq A\leq G$ such that $N$ is pronilpotent of finite rank, $A/N$ is finitely generated abelian, and $G/A$ is finite. Here, $N$ is the {\em pro-Fitting} subgroup of $G$, that is, the group generated by all the subnormal pro-$p$ subgroups of $G$ (over all primes); it is pronilpotent, so closed. Likewise the subgroup $A$ of $G$ has finite index in $G$ and so is open by Theorem~\ref{ns}, and hence is closed. 
Replacing $G$ by a subgroup of finite index if necessary, we may assume $G=A$. 
We aim to show that after a further reduction we have in fact $G=N$
and that this is a Cartesian product of pro-$p_j$ groups $P_j$
 for $j$ in some finite set $J$. In view of the finite rank of $G$  and Theorem~\ref{lazard} (iii)$\Rightarrow$ (i) this
 will prove ( $4.\Rightarrow 1.$).

By Proposition 2.4.3 of \cite{wilson}, the group $G/N$ is a Cartesian  product of groups $\{Q_l:l\in L\}$, where $Q_l$ is an abelian   pro-$r_l$ group and $\{r_l:l\in L\}$ are distinct primes. Let $\pi_N$ denote the map $G\to G/N$. For any proper finite index subgroup $R_l$ of $Q_l$, the group $\pi_N^{-1}(R_l)$ has finite index in $G$, so by Theorem~\ref{ns} and fullness, such groups $\pi_N^{-1}(R_l)$ and their finite intersections are uniformly definable in $\mathcal{G}$. It follows by Lemma~\ref{l:NTP2} (with $H=G$) that the set $L$ is finite.

Again using \cite[Proposition 2.4.3]{wilson} the pronilpotent group $N$ is a Cartesian product of pro-$p_j$ groups $P_j$
 for $j\in J$. We may again assume that the $P_j$ (for $j\in J$) are pairwise distinct primes.

\medskip

{\em Claim } The set $J$ is finite.

{\em Proof.} Suppose that $J$ is infinite. Define
$$J^*:=\{j\in J: \mbox{~for all~} l\in L, r_l\neq p_j\},$$
and let $M$ be the Cartesian product of the $P_j$ for $j\in J^*$. Then the supernatural numbers $|M|$ and $|G/M|$ (see \cite[Section 2.1]{wilson}) are coprime, so by the Schur-Zassenhaus Theorem for profinite groups (\cite[Proposition 2.3.3]{wilson}), there is $B\leq G$ such that $B\cap M=1$ and $G=M\rtimes B$; see Proposition 2.3.3 and the preceding pages of \cite{wilson} for background here.  Each group $P_j$ 
is finitely generated, so has finitely many proper subgroups of each finite index, and so (by considering the intersection of all subgroups of any fixed index greater than 1) has a proper characteristic finite index subgroup $R_j$. 
Let $\pi_j:  M=\Pi_{i\in J^*}P_i\longrightarrow P_j$ denote the canonical projection.
Then the group $\overline{R_j}=\pi_j^{-1}(R_j)\rtimes B$ has
finite index in $G$  and hence the collection of such groups and their finite intersections is uniformly definable in $\mathcal{G}$.
By Lemma~\ref{l:NTP2},  applied to $H=M$ and the $\overline{R_j}$, we see that $J^*$ and consequently $J$ are finite sets.

\medskip

We may now suppose that $N=P_1\times \ldots \times P_t$ where each $P_i$ is a Sylow $p_i$-subgroup of $N$, and 
$G/N =Q_1\times \ldots \times Q_t\times U$, where each $Q_i$ is a (possibly trivial) Sylow $p_i$-subgroup of $G/N$ and $U$ is a direct product of finitely many Sylow subgroups of $G/N$.
As before, let $\pi_N:G\to G/N$ be the natural map and for $B\leq G/N$ let $\overline{B}:=\pi_N^{-1}(B)$. 
By the Schur-Zassenhaus Theorem, since the supernatural numbers $|N|$ and $|U|$ are coprime, we may write $\bar{U}=N\rtimes V$ for some $V\leq G$. By Theorem~\ref{aut}, $V$ induces a finite group of automorphisms on $N$, so, replacing $G$ by a subgroup of finite index if necessary, we may suppose $V$ centralises $N$ and so $\bar{U}=N\times V$. Since $N$ is the maximal pronilpotent subgroup of $G$ this forces $V=1$, so we now assume $U=1$.

Also, for each $i=1,\ldots, t$ let $N_i:=\Pi_{j\neq i} P_j$. Then $N_i\triangleleft \overline{Q_i}$ and $|N_i|$ and $\overline{Q_i}/N_i$ are coprime, so again by the Schur-Zassenhaus Theorem we may write $\overline{Q_i}=N_i\rtimes D_i$ for some pro-$p_i$-group $D_i\leq \overline{Q_i}$. Again, by Theorem~\ref{aut}, each $D_i$ induces a finite group of automorphisms of $N_i$, so, replacing $G$ by a subgroup of finite index if necessary, we may suppose $D_i$ centralises $N_i$, that is, $\overline{Q_i}=N_i\times D_i$ for each $i$. Since each $D_i$ is a pro-$p_i$-group so pronilpotent, it follows that $\overline{Q_i}\leq N$ for each $i$, that is, $G=N=P_1\times \ldots \times P_t$. Finally, as noted above, each $P_i$ has finite rank by Lemma~\ref{finrank}, so we obtain condition (1.) of Theorem~\ref{mainconj} by applying Theorem~\ref{lazard} (iii)$\Rightarrow$ (i) to each $P_i$. $\Box$

\section{Further Observations}

We consider briefly what can be said about NIP profinite groups without the assumption that they are full. It seems difficult to prove any version of Theorem~\ref{mainconj} without the fullness assumption, but we obtain the following. 

\begin{proposition} \label{prosol}
Let $\mathcal{G}=(G,I)$ be a NIP profinite group. Then $G$  has a prosoluble open normal subgroup of finite index.
\end{proposition}

\begin{proof}Let $\{K_i:i\in I\}$ be the family of open subgroups of $G$ indexed by $I$. 
For each $i\in I$ let $N_i:= \bigcap_{g\in G}K_i^g$, which is a finite intersection as $|G:K_i|$ is finite. Hence the $N_i$ are uniformly definable open normal subgroups of $G$.  Setting $H_i:=G/N_i$, the $H_i$ are uniformly interpretable. Put $\mathcal{C}:=\{H_i:i\in I\}$. Then any non-principal ultraproduct of members of $\mathcal{C}$ is definable in an ultrapower of $\mathcal{G}$, so is NIP. It follows by \cite[Theorem 1.2]{mt} that there is $d\in {\mathbb N}$ such that for each $H\in \mathcal{C}$, the soluble radical $R(H)$ of $H$
satisfies $|H:R(H)|\leq d$. Furthermore, by \cite{wilson2}  there is a formula $\phi(x)$ such that for each $H\in \mathcal{C}$, $R(H)=\{x\in H: H\models \phi(x)\}$.

It is now easily checked that $G$ has a definable open normal subgroup $G_1$ such that for all $i\in I$, $G_1/(N_i\cap G_1)$ is soluble. Thus, $G_1$ is prosoluble. 
\end{proof}

It would be interesting to find more examples of NIP 2-sorted profinite groups which are {\em not} full. In particular, Chatzidakis has proposed:

\begin{problem} Find examples of NIP 2-sorted profinite groups $\mathcal{G}=(G,I)$, not $p$-adic analytic,  where $I$ is totally ordered.
\end{problem}

The point here is that a {\em chain} $I$ cannot witness the independence property, and also Lemma~\ref{l:NTP2} should not be applicable.
 The following observation gives natural examples of non-full NIP 2-sorted profinite groups, with $I$ totally ordered, which {\em are} $p$-adic analytic.  
\begin{lemma}
Let $G$ be a uniformly powerful pro-$p$-group, and $\mathcal{G}=(G,I)$ be full. Relabelling elements of $I$, we may suppose that $\omega\subset I$, with $K_n:=P_n(G)$ for each $n\in \omega$. Then  the 2-sorted profinite group $\mathcal{G}^*=(G,\omega)$ is definable in ${\mathbb Q}_p^{{\rm an}}$, so is NIP.
\end{lemma}

\begin{proof} We have $P_n(G):=\{g\in G:\omega(g)\leq n\}$. Thus, it suffices to observe that, by Theorem~\ref{dusrep}(iii), the function $\omega$ is definable in ${\mathbb Q}_p^{{\rm an}}$. 
\end{proof}

We note next the following corollary of the characterisation in Theorem~\ref{mainconj}.

\begin{proposition} \label{strongNIP}
Let $\mathcal{G}=(G,I)$ be a full profinite group with NTP${}_2$ theory $T$. Then
 the closed subgroups of $G$ are uniformly definable in a NIP expansion of $\mathcal{G}$, 

\end{proposition}

\begin{proof}
(i) We use the characterisation of the group $G$ in Theorem~\ref{mainconj} (i). First suppose that $G$ is a powerful pro-$p$ group of rank $d$. For example by \cite[Corollary 5.8]{klopsch}, if  $H$ is a closed subgroup of $G$, then there are $h_1,\ldots,h_d\in H$ such that $H=\overline{\langle h_1\rangle}\cdot \ldots\cdot \overline{\langle h_d\rangle}$. By \cite[Proposition 1.28]{ddms}, 
each subgroup $\overline{\langle h_i\rangle}$ has the form $h_i^{{\mathbb Z}_p}$. Thus, using Proposition~\ref{dusautoy} and in particular Lemma~\ref{dusexp}, the closed subgroups of
 $G$ are uniformly definable in ${\mathbb Q}_p^{{\rm an}}$. For the general case, we now argue as in Section 4, at the end of the proof of Theorem~\ref{mainconj} ($1.\Rightarrow 2$).
\end{proof}

\begin{remark} \rm
1. The proof of Proposition~\ref{prosol} (together with the proof of Theorem 3.1 in \cite{mt}) suggests the following further observation. Let $\mathcal{G}=(G,I)$ be a NIP profinite group. As above, let $\{K_i:i\in I\}$ be the family of open subgroups of $G$ indexed by $I$, and for each $i\in I$ let $N_i:=\bigcap_{g\in G} K_i^g$, and put $H_i:=G/N_i$. Let $\mathcal{C}:=\{H_i:i\in I\}$. Let $L$ be the language of groups, and let $L_P=L \cup\{P_j:j\in I\}$ where each $P_j$ is a unary predicate. View each finite group $H_i$ as an $L_P$-structure, interpreting $P_j$ by $N_iN_j/N_i$. Let $\mathcal{U}$ be a non-principal ultrafilter on $I$, and let $H^*$ be an ultraproduct of the $L_P$-structures $H_i$ with respect to $\mathcal{U}$.

Now each $P_j$ is interpreted in $H^*$ by a finite index subgroup $P_j(H^*)$. Let $P^*:=\bigcap_{j\in I} P_j(H^*)$. Then by compactness and $\omega_1$-saturation of the $L_P$-structure $G^*$, we have $H^*/P^*\cong G$. Furthermore, we may also view each $H_i$ in a natural way as an $L_{\rm prof}$-structure, and hence $H^*$ as an $L_{\prof}$-structure $\mathcal{H}^*=(H^*,I^*)$, where $I\subset I^*$. Clearly $\mathcal{H}^*$ has NIP, being interpretable in an ultrapower of $\mathcal{G}$. Now $\mathcal{H}^*$ has an elementary extension
$(\tilde{H},J)$ containing an element $j\in J$ such that $I=\{i\in J: i<j)$. It follows that $I$ is definable in $(\tilde{H},J)$, so $P^*$ is an externally definable set in $\mathcal{H}^*$ (see \cite[Definition 3.8]{simon}). By a theorem of Shelah (\cite{shelah}, see also \cite[Corollary 3.24]{simon}), the expansion of any NIP structure by the collection of all externally definable sets is NIP. Thus, the expansion $(\mathcal{H}^*,P^*)$ of $\mathcal{H}^*$, in which $P^*$ is named by a unary predicate, itself has NIP theory (though it will not be pseudofinite, due to the definability of $I^*$).

2. If $G$ is a group definable in a model $M$ of a theory $T$, then $G^{\circ}$ is defined to be the intersection of the finite index definable subgroups of $G$.

By an easy consequence of the Baldwin-Saxl Theorem (see \cite[Section 8.1.2]{simon}), if $T$ is NIP then  `$G^{\circ}$ exists', that is, $G^{\circ}$ is type-definable over $\emptyset$ and $|G:G^{\circ}|\leq 2^{|T|}$. Assuming that $G$ is sufficiently saturated, the quotient $G/G^{\circ}$ does not depend on the particular model  $G$, and so is an invariant of $T$, and naturally carries the structure of a profinite group. It can be checked that if $\mathcal{G}=(G,I)$ is a 2-sorted NIP full profinite group with theory $T$, then this invariant quotient
is isomorphic to  $G$ itself; this holds for example because $G$ has finitely many subgroups of index $n$ for each $n$, and they are all definable. 

Recall (see \cite{shelah2}) that if $G^*$ is a sufficiently saturated NIP group then $G$ has a unique smallest type-definable subgroup $(G^*)^{oo}$ of bounded index (that is, $|G:(G^*)^{oo}|<\kappa$, where the underlying structure is $\kappa$-saturated); also  $(G^*)^{oo}$ is type-definable over $\emptyset$ and   the quotient $G^*/(G^*)^{oo}$ carries the `logic topology', whose closed sets are the preimages of type-definable subsets of $G^*$. This is a compact topology on $G^*/(G^*)^{oo}$  (see e.g. \cite[Lemma 8.9]{simon}), and the isomorphism type of $G^{*}/(G^*)^{oo}$ as a topological group is independent of the choice of (sufficiently saturated) $G^*$.  It is shown in \cite{ons} -- see Corollary 2.4 and the remarks before Corollary 2.3 -- that if $G$ is a compact $p$-adic analytic group then $(G^*)^{oo}=(G^*)^o$ (this holds where $G$ is a definable object in the structure ${\mathbb Z}_p^{\an}$, and hence where it is a definable object in the full 2-sorted structure $\mathcal{G}$, since the finite index subgroups are the same). It follows easily  from Theorem~\ref{mainconj} that if $\mathcal{G}=(G,I)$ is {\em any} full profinite NIP group, and $G$ is viewed as a definable group in $\mathcal{G}$, then $(G^*)^{oo}=(G^*)^o$, and the corresponding quotient $G^*/(G^*)^{oo}$ is isomorphic as a topological group to $G$.

3. If $H^*$ is the pseudofinite group arising from a NIP profinite group $\mathcal{G}=(G,I)$ as described in (1), then $(H^*)^{\circ} =P$, so $G\cong H^*/(H^*)^{o}$ is the corresponding invariant for ${\rm Th}(H^*)$ in the language $L_P$.
Working  now in the language $L_{\prof}$, if $\mathcal{G}$ is {\em full}, then $H$ is a quotient of $G^*$ by some definable normal subgroup $K_i$ where $i$ lies in the ultrapower $I^*$ of $I$. 
It is clear from (2) above that $K_i<(G^*)^{oo}$, and hence that $(H^*)^{oo}=(H^*)^o=P$, with $(H^*)^{oo}/H^*\cong G$ (as topological groups, with the connected components in the language $L_{\prof}$). The last assertions also hold in the language $L_P$. 

4. Proposition 2.8 of \cite{ons}, combined with Theorem~\ref{mainconj},  easily yields that if $\mathcal{G}^*=(G^*,I^*)$ is a sufficiently saturated elementary extension of the full profinite NTP${}_2$ group $\mathcal{G}=(G,I)$, then $G^*$ is compactly dominated by $G^*/(G^*)^{oo}$; that is, if $\pi:G^*\to G^*/(G^*)^{oo}$ is the natural map, and $\mathcal{J}$ denotes the ideal of Haar measure zero sets in 
$G^*/(G^*)^{oo}$, then for every definable $X\subset G^*$, 
$$\{x\in G^*/(G^*)^{oo}: \pi^{-1}(x)\cap X\neq \emptyset \wedge \pi^{-1}(x)\cap (G^*\setminus X)\neq \emptyset\} \in \mathcal{J}.$$

\end{remark}

We propose the following speculative conjecture as a version of a model-theoretic one-based/field-like dichotomy conjecture for compact $p$-adic analytic groups. As above, we below view a finite group $G/P_n(G)$ as an $L_{{\rm prof}}$-structure, interpreting $K_i$ for $i\in I$ by the group $K_iP_n(G)/P_n(G)$.

\begin{conjecture}
Let $\mathcal{G}=(G,I)$ be a uniformly powerful pro-$p$ group, full as a profinite group. Then the following are equivalent.

(i) The ring ${\mathbb Z}_p$ is not interpretable in $\mathcal{G}$.

(ii) For every sentence $\sigma$ in the language $L_{{\rm prof}}$, there is $N\in \omega$ such that either each quotient $(G/P_n(G),I)$ satisfies $\sigma$ for $n>N$, or each such quotient satisfies $\neg\sigma$.
 
(iii) The group $G$ is nilpotent-by-finite.
\end{conjecture}

One intuition here is that if $G$ is not nilpotent-by-finite, then we might hope to have $[P_i(G),P_j(G)] =P_{i+j}(G)$, or at least $[P_i(G),P_i(G)]=P_{2i}(G)$,
in which case the index set $I$ inherits some algebraic structure analogous to Presburger arithmetic, suggesting that (ii) is false.

Finally, we ask whether there is any analogue of Theorem~\ref{mainconj} connecting NIP pro-algebraic groups to groups definable in rings such as ${\mathbb C}[[T]]$ over an analytic language.

\end{document}